\documentclass[12pt]{amsart}
\usepackage{enumerate,amsmath,amssymb,bm}
\usepackage{url}
\usepackage{xspace}

\usepackage[margin=1in]{geometry}

\DeclareMathOperator{\pnt}{\raise 0.5mm \hbox{\large\textbf{.}}}

\newcommand{\note}[2][ ]{}

\newtheorem{theorem}{Theorem}
\newtheorem{lemma}[theorem]{Lemma}

\newtheorem{corollary}[theorem]{Corollary}
\newtheorem{conjecture}[theorem]{Conjecture}

\newtheorem{alphconj}{Conjecture}

\theoremstyle{definition}
\newtheorem{definition}[theorem]{Definition} 
\newtheorem{remark}[theorem]{Remark}
\newtheorem{example}[theorem]{Example}

\newtheorem{question}[theorem]{Question}
\newtheorem{problem}[theorem]{Problem}

\usepackage[pdfauthor={William J. Keith and Fabrizio Zanello},
            pdftitle={ft3 and eta powers},
            pdfsubject={enumerative combinatorics},
            pdfkeywords={regular partitions, parity},
            pdfproducer={Latex with hyperref},
            pdfcreator={latex->dvips->ps2pdf},
            pdfpagemode=UseNone,
            bookmarksopen=false,
            bookmarksnumbered=true]{hyperref}

\begin{document}
\title[Parity of the coefficients of certain eta-quotients, III]{Parity of the coefficients of certain eta-quotients, III: two special classes}
\author{William J. Keith and Fabrizio Zanello} \address{Department of Mathematical Sciences\\ Michigan Tech\\ Houghton, MI  49931-1295}
\email{wjkeith@mtu.edu, zanello@mtu.edu}
\thanks{2020 {\em Mathematics Subject Classification.} Primary: 11P83; Secondary:  05A17, 11P84, 11P82, 11F33.\\\indent 
{\em Key words and phrases.} Partition function; density odd values; singular overpartition; partition identity; modular form modulo 2; eta-quotient; eta-power; parity of the partition function.}

\maketitle

\begin{abstract}
We continue a series of papers studying the parity of families of eta-quotients, which provide implications for the parity of the partition function as well as an overarching conjecture on related $q$-series.  The present article focuses on two classes. One consists of eta-quotients of the form $f_t^3/f_1$, a distinguished case of Andrews' singular overpartitions that has recently attracted attention among researchers. In addition, we investigate the parity of certain pure eta-powers $f_1^t$, appending new results to known density theorems.
\end{abstract}

\section{Introduction and Discussion of the Results and Conjectures}


The \emph{density} of a property $P$ of elements of a set $S \subseteq \mathbb{N}$ (or \emph{relative density}, if $S$ is not $\mathbb{N}$) is $$\lim_{n \rightarrow \infty} \frac{1}{\vert \{ s \leq n : s \in S\} \vert} \vert \{ s \leq n : s \in S, s \text{ has property } P\} \vert ,$$ if this limit exists. If the density of a property is 0, the sequence is \emph{lacunary} with respect to the property.  (A lacunary sequence without further specification is a sequence $\{ a_n \} \vert_{n=0}^\infty$ such that the $a_n$ are almost all 0.) Because the existence of many of the limits of interest in this paper has not yet been established, in order to avoid heavy repetitions we may sometimes implicitly assume existence.

Denote $f_t := \prod_{k=1}^\infty (1-q^{kt})$.  Some of the tools available in this subject apply to \emph{eta-quotients}, which are (up to a shift by a power of $q$) finite quotients of the form $\frac{\prod_{i=1}^uf_{\alpha_i}^{r_i}}{\prod_{i=1}^tf_{\gamma_i}^{s_i}}$, where the integers $\alpha_i$ and $\gamma_i$ are positive and all distinct, and $r_i, s_i > 0$.

The \emph{partition function} $p(n)$ counts the number of nonincreasing whole number sequences that sum to $n$.  For example, $\{(4), (3,1), (2,2), (2,1,1), (1,1,1,1)\}$ is the set of partitions of 4, and hence $p(4) = 5$.  The generating function of $p(n)$ is \cite{Andr} $$\sum_{n=0}^\infty p(n) q^n = \frac{1}{f_1}.$$ A longstanding open question regarding $p(n)$, widely believed to be true but considered horrendously difficult, is the \emph{Parkin-Shanks conjecture} \cite{PaSh}: namely, that the density of the odd partition numbers exists and equals 1/2.  This conjecture is the chief motivating question behind this series of papers, as well as much other work in partition theory.  At present, it is not even known that such density exists, nor, assuming it does, that it differs from 0 or 1; i.e., that a positive proportion of the partition numbers is even, or that a positive proportion is odd.

For the remainder of this paper, when we speak of the density of a sequence -- frequently, the coefficients of a power series over an arithmetic progression in the whole numbers -- unless otherwise qualified we refer to the relative density of its odd values. When we say of two power series, $f(q) = \sum_{n=n_1}^\infty a(n) q^n$ and $g(q) = \sum_{n=n_2}^\infty b(n) q^n$, that $f(q) \equiv g(q)$, we will mean $a(n) \equiv b(n) \pmod{2}$ for all $n$.

As a measure of the unpredictability of the parity of $p(n)$, it is known (by Radu \cite{RaduNoEvens}, completing work of Ono \cite{Ono2} and previous authors) that there exists no arithmetic progression $An+B$, $A>0$, such that $p(An+B)$ is always even or always odd.  This is in stark contrast to the guaranteed existence, for any $m \geq 5$, $\gcd(m,6)=1$, of progressions $p(An+B) \equiv 0 \pmod{m}$ (see \cite{Ahlgren1}).

However, for some eta-quotients closely related to the partition function, much more can be said. Further, facts concerning particular eta-quotients potentially have implications for the partition function itself.  Many such implications have been established in the other papers of this series \cite{KZ, KZ2}, as well as previous research with Judge \cite{JKZ, JZ} and other work. For example, in \cite{JKZ}, it was shown that if the $t$-\emph{multipartitions} generated by $\sum_{n=0}^\infty p_t(n)q^n = 1/f_1^t$ have positive density $\delta_t$, for any $t \in \{5,7,11,13,17,19,23,25\}$, then so does $p(n)$ (i.e., $\delta_1 > 0$). This list was later extended by Chen \cite{Chen} and the second author \cite{Zanello} in much greater generality: 

\begin{theorem}[\cite{Zanello}, Theorem 4]\label{ZanDensities} \phantom{.}
\begin{enumerate}
\item If there exists an integer $t \equiv \pm 1 \pmod{6}$ such that $\delta_t > 0$, and all densities $\delta_i$ exist for $i \leq t$, $i \equiv \pm 1 \pmod{6}$, then $\delta_1 > 0$.
\item If there exists an integer $t \equiv 3 \pmod{6}$ such that $\delta_t > 0$, and all densities $\delta_i$ exist for $i \leq t$, $i \equiv 3 \pmod{6}$, then $\delta_3 > 0$.
\end{enumerate}
\end{theorem}

\begin{remark} It perhaps bears some investigation as to why the two cases are so separate, in that it does not seem to follow in the same way that $\delta_{3j} > 0$ implies $\delta_1 > 0$, for any $j$.
\end{remark}

The $t$-\emph{regular} partitions are those in which no part is divisible by $t$.  Their generating function is $$\frac{f_t}{f_1} := \sum_{n=0}^\infty b_t(n) q^n.$$  Denote the density of the odd values of $b_t(n)$ by $\delta^{[t]}$.  For many odd values of $t$, there are known progressions in which $b_t(An+B) \equiv 0 \pmod{2}$, putting an upper limit on $\delta^{[t]}$.  For instance, $b_{5}(10n+6) \equiv 0 \pmod{2}$ and $b_5(10n+8) \equiv 0 \pmod{2}$, implying $\delta^{[5]} \leq 0.8$.

In this paper we consider the parity of the coefficients of two classes of eta-quotients: those of the form $f_t^3 / f_1$, and pure eta-powers $f_1^t$.  The remainder of this introduction will list the theorems to be shown.  A brief list of necessary background facts is given in Section \ref{background}, and the proofs of the claims in the following sections.  We conclude with some remarks and open questions.

We note that all theorems shown in this paper are consistent with our ``master conjecture'' on the parity of eta-quotients, first published in \cite{KZ}, which in essence posits that any given eta-quotient is dissectable into arithmetic progressions whose density is ``either 0 or 1/2.''  We defer its formal statement to Conjecture \ref{mainconj} in the final section.

\subsection{Results for $f_t^3 / f_1$}

For any whole number $t$, write $t=2^\alpha t_0$, with $t_0$ odd. A primary object of interest in this paper is $$G_t(q) = \frac{f_t^3}{f_1} := \sum_{n=0}^\infty g_t(n) q^n.$$  The coefficients of this eta-quotient are equivalent modulo 2 to those of the special case $\overline{C}_{4t,t}$ of the $(k,i)$-\emph{singular overpartitions} whose parity was first studied by Andrews and other authors (\cite{AndrCKI, ArichetaEtAl, BarmanSingh2}).  Denote the density of the odd values of $g_t(n)$ by $\delta^{(t)}$.

In Section \ref{gtproofs} we study the parity of $g_t(n)$ in greater detail. In particular, one of our results in this paper will be a class of relations between certain $\delta^{[t]}$ and $\delta^{(j)}$.  Our main results are the following.

\begin{theorem}\label{smallt0} Let $t = t_0 \cdot 2^\alpha$, with $t_0$ odd.  If $3 \cdot 2^\alpha \geq t_0$, then $\delta^{(t)} = 0$.
\end{theorem}

While we include this here for context, Theorem \ref{smallt0} immediately follows from a result of Cotron \emph{et al.} \cite{CMSZ} (see also \cite{ArichetaEtAl}, Theorem 1).

We next produce a number of arithmetic progressions on which various $g_t(An+B)$ are even.  As discussed before, this puts an upper bound on $\delta^{(t)}$, should such density exist.  

Some progressions have been produced for small $t$, along with other results, by earlier authors (\cite{ArichetaEtAl, BarmanSingh2}).  Naika and Gireesh \cite{NaikaGireesh} proved infinite families of congruences modulo 2 for $g_4$ and $g_7$.  Aricheta \cite{Aricheta2} proved the existence of infinite families of congruences for $g_1$, $g_2$, and $g_3$.  However, it is easy to notice that $g_1$ is of the same parity as $f_2$, $g_2$ of the same parity as $f_1^5 \equiv f_1 f_4$, while $g_3$ is the generating function for the 3-core partitions, which are odd exactly at the octagonal numbers $3m^2-2m$, $m \in \mathbb{Z}$ (Robbins \cite{Robbins}).  An elementary analysis of quadratic residues can therefore identify infinite families of arithmetic progressions avoided by these functions; for instance, $f_2$ has even coefficients at $11n+1$ since $m(3m-1)$ avoids the residue class 1 modulo 11.

The following theorem shows the existence of progressions that are ``inherited'' from known results for the $t$-regular partitions.

\begin{theorem}\label{inherited} We have that $g_t(2tn+B) \equiv 0$ for all $n \geq 0$ for the following values of $t$ and $B$.
{\ }\\
\begin{center}\begin{tabular}{|c|c|}
\hline $t$ & $B$ \\
\hline 5 & 2, 6 \\
\hline 7 & 7, 9, 13 \\
\hline 11 & 2, 8, 12, 14, 16 \\
\hline 13 & 2, 10, 16, 18, 20, 22 \\
\hline 19 & 2, 8, 10, 20, 24, 28, 30, 32, 34 \\
\hline
\end{tabular}
\end{center}
\end{theorem}

{\ }\\
The theorem for $t=5$ and $t=7$ is the $\alpha = 0$ case of results by Pore and Fathima \cite{PoreFathima} and Li and Yao \cite{LiYao}, respectively.  The remainder are new.

Other congruences may be proved as well, frequently by use of the theory of modular forms.  Two examples follow. 

\begin{theorem}\label{g19} $g_{19}(16n+11) \equiv 0 \pmod{2}$.
\end{theorem}

\begin{theorem}\label{g147} $g_{147}(28n+19) \equiv 0 \pmod{2}$.
\end{theorem}

We additionally have the following congruences, which are samples of a collection all of which can be proved with techniques similar to those used in this paper.  We use the common notation $y_a^{-1}$ to mean the inverse of $y$ modulo $a$.

\begin{theorem}\label{longlist} \phantom{.}
\begin{enumerate}
\item If $-6$ is not a quadratic residue modulo a prime $p>3$, then it holds for all $n$ and for any $k \not\equiv 0 \pmod{p}$ that $$g_9(2(p^2n+kp-24_{p^2}^{-1}-1)+1) \equiv 0 \pmod{2}.$$ 
\item If $n$ is not of the form $2(3k^2-2k)$ or $3+11(3k^2-2k)$, then $g_{11}(2n) \equiv 0 \pmod{2}$.  Hence if 22 is a quadratic residue modulo a prime $p$, then at least $(p-1)/2$ even progressions exist of the form $g_{11}(2(p^2n+B)) \equiv 0$.  These occur when $B$ is not of the form $6x^2 - 2 \cdot 3_{p^2}^{-1}$ modulo $p^2$.
\item If $n$ is not of the form $1+15\binom{k+1}{2}+j(3j-1)$, then $g_{15}(4n+2) \equiv 0$.  Hence if $-10$ is not a quadratic residue modulo a prime $p>5$, then for any $k \not\equiv 0 \pmod{p}$ it holds that $$g_{15}(4(p^2n+ k p +1-47 \cdot 24_p^{-1})+2) \equiv 0.$$
\end{enumerate}
\end{theorem}

\begin{example} To illustrate the claims of this theorem, we give the following instances.  An example of clause (1) is that for $k \not\equiv 0 \pmod{13}$, $$g_9(338n+26k+13) \equiv 0.$$  Clause (2) implies many even arithmetic progressions, such as $$g_{11}(118n+k) \equiv 0$$ for $k \in \{2,10,12,\dots,110\}$.  An example of clause (3) is that for $k \not\equiv 0 \pmod{29}$, $$g_{15}(3364n+116k+3222) \equiv 0.$$
\end{example}

The next result provides several interesting identities between $\delta^{(j)}$ and $\delta^{[m]}$, i.e., densities of $(4j,j)$-singular overpartitions and of $m$-regular partitions.

\begin{theorem}\label{equaldensities} The following congruences of $q$-series hold.
\begin{align}G_5 &\equiv f_1^4f_5^2 + q \frac{f_5^8}{f_1^2}; \\
G_5 &\equiv \sum_{n=0}^\infty b_5(2n+1) q^n; \\
G_7 &\equiv f_1^{20} + q f_1^2 f_7^6 + q^6 \frac{f_7^{24}}{f_1^4}; \\
q G_9 &\equiv f_1^2 + \frac{f_3}{f_1}; \\
q G_{10} &\equiv q \frac{f_5^6}{f_1} \equiv f_5 + f_1^5; \\
q^2 G_{18} &\equiv f_1^5 + \frac{f_6}{f_1}.
\end{align}
\end{theorem}

Aricheta \emph{et al.} \cite{ArichetaEtAl} also proved the congruence given in clause (1). The rest are new.

As an immediate corollary, we have the following of equivalences of densities.

\begin{corollary} The following identities hold for the noted densities, should they exist.
\begin{align}\delta^{(5)} &= \frac{1}{2} \delta^{[20]};\\
\delta^{(5)} &= 2 \delta^{[5]};\\
\delta^{(7)} &= \frac{1}{4} \delta^{(14)};\\
\delta^{(9)} &= \delta^{[3]};\\
\delta^{(10)} &= 0;\\
\delta^{(18)} &= \delta^{[6]}.
\end{align}
\end{corollary} 

\begin{remark} \phantom{.}
\begin{enumerate}
\item For many of these density implications, if one density exists, the other does as well, since the remaining terms in the identities of Theorem \ref{equaldensities} are lacunary.
\item The density implication (7) of clause (2) requires the additional fact that $b_5(2n)$ is lacunary modulo 2.  Clause (12), following from (6), uses the lacunarity of $f_1^5$, which follows from \cite{CMSZ} (see Theorem \ref{cot} below).  Indeed, due to this lacunarity, clause (6) gives more, namely that $g_{18}(n-2)$ is almost always of the same parity as $b_6(n)$.
\item We conjecture that $\delta^{[20]} = 1/2$, and that $\delta^{[3]} = 1/4$.  Also, since $\delta^{(t)}$ is at most 1, we have $\delta^{(7)} \leq 1/4$ (always assuming existence).
\item That the density $\delta^{(10)}$ exists, and is exactly 0, does follow from Theorem \ref{equaldensities}, and also from Theorem \ref{smallt0}.
\end{enumerate}
\end{remark}

\begin{theorem}\label{ThreeT} Let $ t = 2^\alpha t_0$, with $t_0$ odd.  If $\frac{t_0}{3} < 2^\alpha$, then assuming existence, $\delta^{(9t)} = \delta^{[3t]}$. Moreover, it almost always holds that $g_{9t}(n) \equiv b_{3t}(n+t)$.
\end{theorem}

In particular, if $\alpha > 0$, $\frac{t_0}{3} < 2^\alpha < 3t_0$, we conjectured that $\delta^{(9t)} = \frac{1}{2}$ (\cite{KZ2}, Conjecture 7).

\begin{theorem}\label{FiveT} Let $ t = 2^\alpha t_0$, with $t_0$ odd and $\alpha > 0$ (hence $t$ is even).  If $t_0 / 3 < 2^\alpha$, then assuming existence, the density $\delta^{(5t)}$ is the same as the density of the odd coefficients of $\frac{f_{t/2}f_{5t/2}}{f_1}$. Moreover, it almost always holds that $g_{5t}(n)$ has the same parity as the coefficient of $q^{n+t/2}$ in the latter series.
\end{theorem}

\subsection{Results for $f_1^t$} 

In  Section \ref{etaproofs}, we present proofs of some novel results on the pure eta-powers $f_1^t$.  As these, along with the multipartitions $1/f_1^t$, may reasonably be considered fundamental building blocks of the subject, it seems wise to include them in this series.  We add a few results to previously known facts, including relations between $G_t(q)$ and $f_1^s$ for various values of $t$ and $s$.  We list here the theorems to be proved.

 For $t$ a whole number, write $$f_1^t = \sum_{n=0}^\infty c_t(n) q^n.$$  Results for $c_t(n)$, combined with identities such as those of Theorem \ref{equaldensities}, then give further properties of the parities of the $G_j$.  For instance, clause (3) of Theorem \ref{equaldensities} yields:

\begin{corollary} $g_7(4n) \equiv c_5(n).$
\end{corollary} 
\noindent
Thus, any parity result satisfied by $c_5(n)$ is inherited, at one-quarter density, by $g_7(4n)$.

Note that if $t = \sum 2^{i_j}$ is the binary expansion of $t$, then $$f_1^t \equiv f_{2^{i_1}} f_{2^{i_2}} f_{2^{i_3}} \dots $$
The greater the number of $i_j$ involved, the more difficult the theorems seem to become.

Our next two results are simple consequences of progressions forbidden by the pentagonal or triangular numbers.  We record them explicitly for completeness.  They are also easy to derive from Corollaries 2 and 3 in \cite{Chen2}. Each of them also yields the parity of the coefficients $c_{t}(n)$ for various classes of even $t$ by magnification; namely, in those instances where there is exactly one $i_j$, or only $i_j$ and $i_j+1$. We will not further repeat this fact.

\begin{theorem}\label{onepent}
For $m$ coprime to 6, we have $c_1(m n+ B) \equiv 0$ whenever $2 \cdot 3_m^{-1} B + 36_m^{-1}$ is not a quadratic residue mod $m$.
\end{theorem}

\begin{theorem}\label{onetri}
For $m$ coprime to 6, we have $c_3(m n+ B) \equiv 0$ whenever $2 B + 4_m^{-1}$ is not a quadratic residue mod $m$.
\end{theorem}

We make the following definition, which concerns a behavior common to not only eta-powers but also eta-quotients such as those studied in the previous papers of this series.

\begin{definition}
We say that $f_1^t$ is $p^2$-\emph{even} at a prime $p$ with base $r \in \{0, \dots, p^2-1\}$ if $c_t( p^2 n + k p + r) \equiv 0$ for all $k \in \{1, \dots, p-1\}$.  We will denote this more compactly by saying that $f_1^t$ is $(p,r)$-even.
\end{definition}

\begin{remark} This behavior is not exclusive to the power series $f_1^t$, so the definition above could certainly be used for other series.
\end{remark}

\phantom{.}

We next offer the following two conjectures.

\begin{alphconj}\label{allt}
For any given $t \geq 1$ odd, $f_1^t$ is $(p,r)$-even for a positive proportion of primes $p$, for some base $r$ depending on $t$ and $p$.
\end{alphconj}

\begin{alphconj}\label{allprimes}
For any given prime $p$, there exist infinitely many $t \geq 1$ odd such that $f_1^t$ is $(p,r)$-even, for some base $r$ depending on $t$ and $p$.
\end{alphconj}

We establish Conjecture \ref{allt} for $t$ a sum of two quadratic terms, and Conjecture \ref{allprimes} for all primes other than 2 or those congruent to 1 mod 24.  

Corollaries 5 and 6 in \cite{Chen2} easily yield the cases in Theorem \ref{twopowers} of the form $a + b \cdot 2^e$ in which at least one of $a$ or $b$ is 3. These are also parts 1, 3, and 4 of Theorem \ref{primeclasses}.  We include a proof of part 2, and in Theorem \ref{primeclasses} we additionally provide the explicit values of the prime classes and progressions arising.  

\begin{theorem}\label{twopowers}
Let $t = a + b \cdot 2^e$, $a, b \in \{1,3\}$, $e > 0$.  Then $f_1^t$ is $(p,r)$-even for some base $r$, when $-2^{e}$ is a quadratic nonresidue modulo $p$ if $a=b$, and when $-3 \cdot 2^{e}$ is a nonresidue if $a \neq b$.  In particular, $f_1^t$ is $p^2$-even for a set of positive density in the primes, if such density exists.
\end{theorem}

\begin{theorem}\label{primeclasses}
We have the following cases of $t$ as the sum of two quadratic terms.
\begin{itemize} 
\item[1.$ \, $] If $t = 2^d+3$, $p \equiv 23 \pmod{24}$ prime, then $f_1^t$ is $(p,r)$-even, with $r \equiv - \left(2^{d-3} 3_{p^2}^{-1} + 2_{p^2}^{-3}\right) \pmod{p^2}$.
\item[$1^\prime$.] If $d$ is even in the previous clause, we may take $p \equiv 5 \pmod{6}$.  If $d$ is odd we may instead take $p \equiv 13 \pmod{24}$.
\item[2.$ \, $] If $t = 2^d + 1$, $p \equiv 7 \pmod{8}$ prime, then $f_1^t$ is $(p,r)$-even, with $r \equiv -3\left(2^{d-3} + 2_{p^2}^{-3}\right) \pmod{p^2}$.
\item[$2^\prime$.] If $d$ is even in the previous clause, we may take $p \equiv 3 \pmod{4}$, $p \geq 7$.
\item[3.$ \, $] If $t = 3 \cdot 2^d + 1$, $p \equiv 23 \pmod{24}$ prime, then $f_1^t$ is $(p,r)$-even,  with \\$r \equiv -\left(2^{d-3} + 2_{p^2}^{-3} \cdot 3_{p^2}^{-1}\right) \pmod{p^2}$.
\item[$3^\prime$.] If $d$ is even in the previous clause, we may take $p \equiv 5 \pmod{6}$.  If $d$ is odd we may instead take $p \equiv 13 \pmod{24}$.
\item[4.$ \, $] If $t = 3 \cdot 2^d + 3$, $p \equiv 7 \pmod{8}$ prime, then $f_1^t$ is $(p,r)$-even with $r \equiv - \left(2^{d-3} + 2_{p^2}^{-3}\right) \pmod{p^2}$.
\item[$4^\prime$.] If $d$ is even in the previous clause, we may take $p \equiv 3 \pmod{4}$.
\end{itemize}
When $d-3 < 0$, the resulting inverse is also taken modulo $p^2$.
\end{theorem}

A brief inspection shows that the clauses above cover all odd primes other than those congruent to 1 mod 24.  Hence we have the following.

\begin{corollary}\label{fixedprime}
Conjecture \ref{allprimes} holds for all primes $p \not\equiv 1 \pmod{24}$, $p \geq 3$.
\end{corollary}

Theorem \ref{primeclasses} applies for $p=5$ in the case $1^\prime$ and $p=3$ in the case $4^\prime$.  In those cases, we observe the following stable behavior of $r$.

\begin{corollary}\label{p5cor}
For $p=5$, $t = 2^{2k}+3$, $d \geq 1$, we have $r \equiv t \pmod{25}$.  For $p=3$, $t = 3 \cdot 2^{2k} + 3$, $d \geq 1$, we have $r \equiv t/3 \pmod{9}$ (and the latter will always be 2 mod 3).
\end{corollary}

The first example of Corollary \ref{p5cor} is that $c_{15}(9n+2)$ and $c_{15}(9n+8)$ are always even.

\begin{example}
As a consequence of the proofs of the above theorems, we can obtain a large class of new congruences, of which the following is the barest sample.  Here $t$ is the power, $p$ a relevant prime from Theorem \ref{primeclasses}, and $r$ the base given in the theorem, so that $c_t(p^2n + k p + r)$ is even for all $k \in \{1, \dots , p-1\}$.  (For instance, the first line in the left table says that $c_7(25n+B)$ is even for $B \in \{2, 12, 17, 22\}$.)

\phantom{.}

\begin{center}\begin{tabular}{rcl}\begin{tabular}{|c|c|c|}
\hline $t$ & $p$ & $r$ \\
\hline 7 & 5 & 7  \\
\hline 17 & 7 & 34  \\
\hline 193 & 5 & 18  \\
\hline 195 & 3 & 2  \\
\hline 
\end{tabular} & \begin{tabular}{|c|c|c|}
\hline $t$ & $p$ & $r$ \\
\hline 7 & 23 & 154 \\
\hline 17 & 11 & 85 \\
\hline 193 & 47 & 84 \\
\hline 195 & 71 & 622 \\
\hline 
\end{tabular}
\end{tabular}
\end{center}
\end{example}
{\ }\\
Note that theorems regarding powers $t$ that require more terms seem to rapidly increase in difficulty. The following results, the last of which is due to Chen, rely on the theory of modular forms.

\begin{theorem}\label{dminus1}
There exists no progression $An+B$, $A \neq 0$, for which $c_{2^d-1}(An+B) \equiv 0$ for all $d$.
\end{theorem}

\begin{theorem}\label{t21}
We have $c_{21}(49n+k) \equiv 0 \pmod{2}$ for $k \in \{14, 28, 35\}$.
\end{theorem}

\begin{theorem}[Chen \cite{Chenktuple}]\label{chenthm}
Let $k=2^r s$ with $s$ odd, say $s = \sum_{i=0}^\infty \beta_i 2^i$ where $\beta_i \in \{0,1\}$.  Let $g_s = 1+ \sum_{i=0}^\infty \beta_{2i+1} 2^i + \sum_{j=0}^\infty \beta_{2j+2} 2^j$, and assume $m \geq g_s$ is an integer. Then, for any distinct odd primes $\ell_1, \dots , \ell_m$, we have
$$c_{3k} \left( \frac{2^r \ell_1 \ell_2 \cdots \ell_m n - k}{8} \right) \equiv 0,$$ for all $n$ coprime to $\ell_1 \ell_2 \cdots \ell_m$ and satisfying $\ell_1 \ell_2 \cdots \ell_m n \equiv s \pmod{8}$.
\end{theorem}
\noindent
(It should be noted that Theorem \ref{t21} does \emph{not} follow from Theorem \ref{chenthm}.)

\section{Background}\label{background}

We will use without frequent comment the identities $f_{2t} \equiv f_t^2$ for $t \in \mathbb{N}$, and $-x \equiv x \pmod{2}$.  The former allows us to move factors of 2 between exponents and bases in $f_t^d$, and the latter to disregard signs in summations.

The following are the mod 2 reductions of Euler's {pentagonal number theorem} and a well-known identity of Jacobi, respectively (see \cite{Andr}, Equations (1.3.1) and (2.2.13)).

\begin{theorem} \begin{align*}f_1 &\equiv \sum_{m=-\infty}^\infty q^{(m/2)(3m-1)} = 1 + q + q^2 + q^5 + q^7 + \dots ;\\
f_1^3 &\equiv \sum_{n=0}^\infty q^{\binom{n+1}{2}}.\end{align*}
\end{theorem}

Since the integers represented by either identity are given by quadratic forms, after completing the square an equivalent description for the eta-powers portion of this paper might have been, ``Parity of the number of representations of integers in arithmetic progressions by certain quadratic forms in two or more variables.''

By completing squares, the exponents appearing in the two expressions can be given the following forms modulo $m$, for $m$ coprime to 6 and 2, respectively. 
\begin{align*}
\frac{n}{2}(3n-1) &\equiv 2_m^{-1} \cdot 3\left(n-6_m^{-1}\right)^2 - 24_m^{-1} \pmod{m}; \\
\binom{n+1}{2} &\equiv 2_m^{-1} \left(n+2_m^{-1}\right)^2 - 8_m^{-1} \pmod{m}.
\end{align*}

Here inverses are being taken modulo $m$, hence the necessity for $m$ to be coprime to 6 in the first line, and 2 in the second.

A recent result of Cotron, Michaelsen, Stamm, and Zhu \cite{CMSZ} (hereinafter CMSZ) established lacunarity for an important class of eta-quotients, building on and generalizing seminal works by Serre \cite{Serre} and then Gordon and Ono \cite{GordonOno}. We phrase it in the following terms:
\begin{theorem}[CMSZ \cite{CMSZ}, Theorem 1.1]\label{cot}
Let $F(q)=\frac{\prod_{i=1}^uf_{\alpha_i}^{r_i}}{\prod_{i=1}^tf_{\gamma_i}^{s_i}}$, and assume that
$$\sum_{i=1}^u \frac{r_i}{\alpha_i} \ge \sum_{i=1}^t s_i\gamma_i.$$
Then the coefficients of $F$ are lacunary modulo 2.
\end{theorem}

Another classical fact that will be useful in this paper is that the product of two quadratic series is lacunary over $\mathbb{Z}$ (hence, in particular, modulo 2).

\begin{theorem}[Landau \cite{Landau}]\label{Landaulac} Let $a(n)$ and $b(n)$ be nondegenerate quadratic forms in the variable $n$.  Then $$\left( \sum_{n=n_0}^\infty q^{a(n)} \right) \left( \sum_{n=n_1}^\infty q^{b(n)} \right)$$ is lacunary.
\end{theorem}

For example, $f_1^4 f_7^{6}$ is lacunary mod 2, since the first factor is equivalent modulo 2 to $f_4$, with exponents $4(m/2)(3m-1)$, and the latter to $\sum_{n \geq 0} q^{14 \binom{n+1}{2}}$.

\begin{lemma}\label{f1f3-f1f5}  The following congruences hold: \begin{align}f_1^3 &\equiv f_3 + q f_9^3; \label{f13clause} \\ f_1 f_5 &\equiv f_1^6 + q f_5^6; \label{f1f5clause} \\ f_1 f_7 &\equiv \frac{f_1^{22}}{f_7^2} + q f_1^4 f_7^4 + q^6 \frac{f_7^{22}}{f_1^2} \label{f1f7clause}.\end{align}
\end{lemma}

\begin{proof} For the first two claims, see \cite{Hirsch2}, Chapter 1 and \cite{Hirsch}, Equation (13), respectively. The third clause appears as Lemma 3.14 in \cite{Xia}.\end{proof}

The following results concerning quadratic residues can be found in any standard number theory text (see, e.g., \cite{Rosen}).

\begin{lemma}\label{quadres} Let $p$ be an odd prime.  Then: 
\begin{itemize}
\item 2 is a quadratic residue modulo $p$ if and only if $p \equiv 1, 7 \pmod{8}$.
\item 3 is a quadratic residue modulo $p$ if and only if $p \equiv 1, 11 \pmod{12}$.
\item $-1$ is a quadratic residue modulo $p$ if and only if $p \equiv 1 \pmod{4}$.
\end{itemize}
\end{lemma}

\section{Proofs for $g_t(n)$}\label{gtproofs}

\begin{proof}[Proof of Theorem \ref{smallt0}] This result follows immediately from CMSZ, Theorem \ref{cot}.  It was also recorded in \cite{ArichetaEtAl}.
\end{proof}

\begin{proof}[Proof of Theorem \ref{inherited}] The listed progressions are all known to be even for $t$-regular partitions with progression modulus $2t$ (see the table in \cite{KZ}). Any such progression is inherited by $g_t(n)$.  Since $$G_t(q) \equiv \frac{f_t}{f_1} \cdot f_{2t},$$ it holds that if the $t$-regular partitions possess any even progressions of the form $b_t(2tn+j)$, then $g_t(2tn+j)$ is also even, since the elements of the latter sequence can be written as a recurrence in the elements of the former.  Using the pentagonal number theorem and equating powers, we obtain $$g_t(n) = b_t(n) - b_t(n-2t) - b_t(n-4t) + b_t(n-10t) + b_t(n-14t) - \dots $$ 
\end{proof}

\begin{proof}[Proof of Theorem \ref{g19}] This is proved by an appeal to the theory of modular forms, for which we direct the reader to any standard text or to the earlier papers in this series \cite{KZ,KZ2}. We provide here a sketch of the proof.

We construct the eta-quotient $$D(q) = q^{21} \frac{f_8^{18}f_{16}^{19} f_{38} f_{19}}{f_1} =: \sum_{n=21}^\infty d(n) q^n.$$  The reader will note that the parity of $d(n)$ can be written as a recurrence in the values $g_{19}(n-16k-21)$: $$D(q) \equiv G_{19}(q) \cdot q^{21} f_{16}^{28}.$$  Hence, $g_{19} (16n+11) \equiv 0 \pmod{2}$ for all $n$ if and only if the same is true for $d(16n)$.

Standard theorems yield that $D(q)$ is a holomorphic modular form of weight 252 on the group $\Gamma_0(2432)$ with character $\chi(d) = \left( \frac{2^{126}19^2}{d} \right)$.  The Sturm bound for this space is 80640.  Since $16 \, \vert \, 2432$, $$D(q) \vert U(16) = \sum_{n=0}^\infty d(16n) q^n$$ is also a modular form in this space with the same character.

We have that $16 \cdot 80640 = 1290240$.  Thus, if $g_{19}(16n+11) \equiv 0 \pmod{2}$ for $n \leq 80640$, which we will determine by calculating $g_{19}(n)$ for all $n\le1290251$, then $g_{19}(16n+11) \equiv 0 \pmod{2}$ for all $n$.

The required calculation was performed with \emph{Mathematica} on a desktop computer, and the necessary data was confirmed, proving the claim.
\end{proof}

\begin{proof}[Proof of Theorem \ref{g147}] This theorem is shown analogously to the previous one, hence we omit the details.  The necessary form is $$D_2(q) = q^{65} \frac{f_{147}^3 f_{28}^{40}}{f_1} \vert U(7) \vert U(4).$$
\end{proof}

\begin{proof}[Proof of Theorem \ref{longlist}] 
Aricheta \cite{ArichetaEtAl} established, among other facts, the following identities: 
\begin{align*}
\sum_{n=0}^\infty g_9(2n+1) q^n &\equiv f_1 \frac{f_9^3}{f_3}; \\
\sum_{n=0}^\infty g_{11}(2n) q^n &\equiv \frac{f_6^3}{f_2} + q^3 \frac{f_{33}^3}{f_{11}}; \\
\sum_{n=0}^\infty g_{15}(4n+2)q^n &\equiv q f_{15}^3 f_2. \\
\end{align*}

An analysis of the quadratic forms in the exponents yields the claim.  Briefly, one completes squares modulo $p^2$, and then finds that certain progressions must be avoided when some value is not a quadratic residue modulo $p$.  When multiple summands are given, one determines progressions avoided by each term.

For $t=9$ and given a prime $p>3$, we see that $g_9(2n+1)$ has an odd exponent only if $n$ can be written in the form $$n \equiv \frac{k}{2}(3k-1) + 3 (3j^2 - 2j) \equiv 3 \cdot 2^{-1} x^2 + 9 y^2 - 24_p^{-1} - 1\pmod{p},$$ for some integers $x, y$.  

When more specifically we have $n \equiv -24_{p^2}^{-1} - 1 \pmod{p^2}$, it may hold that there is a solution with $x, y \equiv 0 \pmod{p}$, but if either $x$ or $y$ is nonzero modulo $p$, then both must be.  In that case, we may write $$\left(\frac{x}{y}\right)^2 \equiv -6 \pmod{p}.$$ But this congruence has no solutions if $-6$ is a quadratic nonresidue modulo $p$.

The $t=11$ clause of the theorem is proved by individual analysis of the two terms of the second identity above; in this case, we find that a representable $n$ must be of the form $2(3k^2-2k) \equiv 6x^2 - 2 \cdot 3_{p^2}^{-1}\pmod{p^2}$, or $33 y^2 - 2 \cdot 3_{p^2}^{-1} \pmod{p^2}$.  But if 22 is a quadratic residue modulo $p$, then 2 and 11 are both simultaneously quadratic residues or nonresidues, and so the sets $\{2x^2\}$ and $\{11y^2\}$ coincide.  But then, since each set contains $(p+1)/2$ elements, the other residues must be missed.

The third ($t=15$)  clause of the theorem is proved completely analogously to the first.
\end{proof}

\begin{remark} In the $t=11$ clause, if 22 is not a quadratic residue modulo $p^2$, then the two sets of values $6x^2 - 2 \cdot 3_{p^2}^{-1}$ and $33 y^2 - 2 \cdot 3_{p^2}^{-1}$ are disjoint except at $-2 \cdot 3_{p^2}^{-1}$, meaning they represent all residue classes modulo $p$ separately.  This argument then cannot establish any congruences of the same sort at all.  However, further congruences almost certainly exist and can likely be shown by additional analysis of the separate collections.
\end{remark}

\begin{proof}[Proof of Theorem \ref{equaldensities}] The first two clauses, for $G_5$, follow from multiplying through (\ref{f1f5clause}) by $f_5^2/f_1^2$ and $1/f_1^2$, respectively.  In the former case, we obtain $$\frac{f_5^3}{f_1} \equiv f_1^4 f_5^2 + q \frac{f_5^8}{f_1^2}.$$  The term on the left is $G_5(q)$.  The first term on the right is lacunary by Landau's lacunarity theorem (Theorem \ref{Landaulac}), and the third term is equivalent to $\frac{f_{40}}{f_2}$, the $q \rightarrow q^2$ magnification of $\frac{f_{20}}{f_1}$.  In the latter case, $$\frac{f_5}{f_1} \equiv f_1^4 + q \frac{f_5^6}{f_1^2}.$$  The term on the left is the generating function for the 5-regular partitions.  The first term on the right is equivalent to $f_4$, and hence immediately lacunary since it is a quadratic series.  The second term on the right is a shift and magnification of $G_5(q)$: $$q \frac{f_5^6}{f_1^2} \equiv q G_5(q^2).$$  The claim follows from extracting the odd-power terms on both sides.  By extracting even powers one immediately obtains the well-known fact that $b_5(2n)$ is lacunary, since it is odd if and only if $n$ is twice a pentagonal number.

The fifth clause of the theorem likewise follows by multiplying through identity (\ref{f1f5clause}) by $1/f_1$.  We obtain $$f_5 \equiv f_1^5 + q \frac{f_5^6}{f_1}.$$  The second term on the right is equivalent to $q G_{10}(q)$.  Since $f_1^5 \equiv f_1 f_4$ is lacunary by Theorem \ref{Landaulac} and $f_5$ is immediately lacunary, and the sum of two lacunary series is lacunary, the claim follows.

The third clause of the theorem follows from multiplying through identity (\ref{f1f7clause}) by $f_7^2/f_1^2$.  We obtain, as claimed, $$\frac{f_7^3}{f_1} \equiv f_1^{20} + q f_1^2 f_7^6 + q^6 \frac{f_7^{24}}{f_1^4}.$$  The density result follows since $f_1^{20} \equiv f_{16} f_4$ and hence the first two terms are lacunary, while the third term is equivalent to $q^6$ times the $q \rightarrow q^4$ magnification of $G_{14}(q)$.

The fourth clause of the theorem is proved by multiplying through identity (\ref{f13clause}) by $1/f_1$ and rearranging terms.  We obtain $$q \frac{f_9^3}{f_1} \equiv f_1^2 + \frac{f_3}{f_1}.$$ The left side is $q G_9(q)$, the first term on the right is lacunary, while the second term generates the 3-regular partitions.  Hence the claim follows.

Finally, the last clause of the theorem is shown using the substitution $q \rightarrow q^2$ in the first clause of Lemma \ref{f1f3-f1f5}, which yields $$f_1^6 \equiv f_3^2 + q^2 f_9^3.$$  Dividing through by $f_1$ and contracting the square on the terms on the right now gives $$f_1^5 \equiv \frac{f_6}{f_1} + q^2 \frac{f_{18}^3}{f_1},$$ which proves the claim.
\end{proof}

\begin{remark}
Any other identity of the type in Lemma \ref{f1f3-f1f5} could be similarly manipulated to obtain implications analogous to those presented here. We also remark that one can easily obtain many congruences satisfied by $g_{t}$ from the identities thus proved.  For instance, $g_5(2n)$ must be even unless $2n$ is 4 times a pentagonal number plus 10 times a pentagonal number, a result that implies numerous even linear progressions by analysis of quadratic residues.
\end{remark}

\begin{proof}[Proof of Theorem \ref{ThreeT}]  By making the substitution $q \rightarrow q^t$ in identity (\ref{f13clause}) and dividing through by $f_1$, we obtain the expansion $$G_t = \frac{f_t^3}{f_1} \equiv \frac{f_{3t}}{f_1} + q^t \frac{f_{9t}^3}{f_1}.$$

If $2^\alpha > t_0 / 3$, then $\delta^{(t)} = 0$ by CMSZ.  Thus, the left side of the congruence, $f_t^3/f_1$, has density 0.  It follows that $\delta^{(9t)} = \delta^{[3t]}$ (assuming existence), and we almost always have that $g_{9t}(n) \equiv b_{3t}(n+t)$, as desired. Note that if, moreover, $2^\alpha > 3t_0$, then by CMSZ all three terms of the congruence have density 0. 
\end{proof}


\begin{proof}[Proof of Theorem \ref{FiveT}]  The logic is similar to the previous theorem.  From Lemma \ref{f1f3-f1f5}, we obtain the congruence $$f_1 f_5 \equiv f_1^6 + q f_5^6.$$ Since $t$ is even, make the substitution $q \rightarrow q^{t/2}$, and replace $f_{t/2}^6$ by $f_t^3$ and $f_{5t/2}^6$ by $f_{5t}^3$.  Isolate $f_{t}^3$ and divide through by $f_1$.

This yields $$G_t \equiv q^{t/2} \frac{f_{5t}^3}{f_1} + \frac{f_{t/2} f_{5t/2}}{f_1}.$$  Therefore, when $t_0 / 3 < 2^\alpha$, the left side has density 0 by CMSZ. Hence the coefficients of $q^{t/2} G_{5t}$ are almost always of the same parity as those of $\frac{f_{t/2}f_{5t/2}}{f_1}$, and the theorem follows. \end{proof}

We propose the following conjecture.

\begin{conjecture}\label{nontrivfivehalves} If $2^\alpha < 5t_0/3$, $\alpha > 0$, then both $$\frac{f_{5t}^3}{f_1} {\ }{\ }\textnormal{and}{\ }{\ } \frac{f_{t/2} f_{5t/2}}{f_1}$$ have density $1/2$. \end{conjecture}

\section{Eta-power proofs}\label{etaproofs} We now include proofs of our eta-power results.  A fact stated but not proved is either trivial or in previous literature.

We begin with a short proof of Theorem \ref{onepent}, which does not require the theta-function machinery of Chen's argument in \cite{Chen2}.  A similar proof can be written for Theorem \ref{onetri}.

\begin{proof}[Proof of Theorem \ref{onepent}]
We have
$$f_1 = \sum_{n=0}^\infty c_1(n) q^n \equiv \sum_{n \in \mathbb{Z}} q^{\frac{n}{2}(3n-1)}.$$
Since for $m$ coprime to 6 we may write
$$\frac{n}{2}(3k-1) \equiv 2_m^{-1} \cdot 3 \left(k-6_m^{-1}\right)^2 - 24_m^{-1} \pmod{m},$$
should $mn+B$ not be of this form for any $n$, it will then hold that $c_1(mn+B)$ is identically 0 mod 2 on the progression. Solving throughout for the square, we obtain:
\begin{align*} mn+B &\equiv 2_m^{-1} \cdot 3 \left(k-6_m^{-1}\right)^2 - 24_m^{-1} \pmod{m} ; \\
2 \cdot 3_m^{-1} B + 36_m^{-1} &\equiv \left(k-6_m^{-1}\right)^2 \pmod{m}.
\end{align*}

Therefore, if $2 \cdot 3_m^{-1} B + 36_m^{-1}$ is not a quadratic residue mod $m$, $f_1$ is identically even on $mn+B$. 

Conversely, if $2 \cdot 3_m^{-1} B + 36_m^{-1}$ is a quadratic residue mod $m$, then there exists some $k-6_m^{-1}$ for which the congruence holds. Thus, $c_1(mn+B)$ is not identically 0 on the progression.
\end{proof}

We include here a simple proof of Theorem \ref{twopowers}, and specifically clause 2 of Theorem \ref{primeclasses}.  We also give an explicit description of the primes arising in Chen's theorems in \cite{Chen2}.

\begin{proof}[Proof of Theorem \ref{twopowers}]
The numbers $t = a + b \cdot 2^e$, $a, b \in \{1,3\}$, are precisely the odd values of $t$ for which we may write
$$f_1^t \equiv f_{1 \text{ or } 3} \cdot f_{1 \text{ or 3}}^{2^e}.$$

Therefore, $t$ can only be odd if it is representable as the sum of two quadratics, one either the pentagonal or the triangular numbers, and the other (independently) the $2^e$-magnified pentagonal or triangular numbers.  (Or an odd number of the unmagnified quadratic, but here we focus on those $t$ that can be written with one of each.)

Suppose $a=b=1$, so that $t = 2^e+1$.  Then we have two pentagonal progressions, one magnified, and the terms $N$ appearing with nonzero coefficient in their product must satisfy, for some $k_1, k_2 \in \mathbb{Z}$,
$$2_{p^2}^{-1} \cdot 3 \left(k_1-6_{p^2}^{-1}\right)^2 - 24_{p^2}^{-1} + (2^e) \left( 2_{p^2}^{-1} \cdot 3 (k_2-6_{p^2}^{-1})^2 - 24_{p^2}^{-1} \right) \equiv N \pmod{p^2}.$$

Let $N$ be of the form $p^2 + kp - 24_{p^2}^{-1} - 2^e 24_{p^2}^{-1},$ $0 < k < p$.  Simplify the notation by letting $x = k_1 - A$ and $y = k_2 - A$, where $A$ is a suitable integer representing $6_{p^2}^{-1}$.  Then we have: 
\begin{align*}
2_{p^2}^{-1} \cdot 3 x^2 - 24_{p^2}^{-1} + (2^e) \left( 2_{p^2}^{-1} \cdot 3 y^2 - 24_{p^2}^{-1} \right) &\equiv kp- 24_{p^2}^{-1} - 2^e 24_{p^2}^{-1} \pmod{p^2} ; \\
2_{p^2}^{-1} \cdot 3 x^2 + \left(2^{e-1}\right) \cdot 3 y^2  &\equiv kp  \pmod{p^2} .
\end{align*}

Now $2_{p^2}^{-1} \cdot 3 x^2 + \left(2^{e-1}\right) \cdot 3 y^2 \equiv 0 \pmod{p}$, so if $x$ or $y$ is divisible by $p$, the other must be as well, but then $2^{-1} \cdot 3 x^2 + \left(2^{e-1}\right) \cdot 3 y^2 \equiv 0 \pmod{p^2}$, which is false since $p \nmid k$.  Hence $x, y \not\equiv 0 \pmod{p}$, and we may write
$$\left( \frac{x}{y} \right)^2 \equiv -2^e \pmod{p}.$$ 

It follows that if $-2^e$ is not a quadratic residue modulo $p^2$, the chosen arithmetic progression $p^2 + kp - 24_{p^2}^{-1} - 2^e 24_{p^2}^{-1}$ with $0 < k < p$ cannot have nonzero coefficients mod 2 in $f_1^t$.

The other cases are similar (and as we have mentioned, are also proved in Chen \cite{Chen2}).  If $a = b = 3$, there is no factor of 3 to cancel in the first place; if $a \neq b$, then we have one, which by properties of quadratic residues may be 3 or $3^{-1}$, without loss of generality.  Only the resulting $r$ is different.  \end{proof}

\begin{proof}[Proof of Theorem \ref{primeclasses}]
We prove clause 1 and $1^{\prime}$.  The arguments for the remaining cases are similar.

Let $t = 2^d + 3$.  We can easily calculate that $-\left(2^{d-3}3_{p^2}^{-1} + 2^{-3}\right) \pmod{p^2}$ is the necessary value of $r$.  We are concerned with when $-3 \cdot 2^d$ is a quadratic nonresidue modulo $p$.

If $p \equiv 23 \pmod{24}$, then 2 is a quadratic residue modulo $p$, as is 3, but $-1$ is a quadratic nonresidue, so $-3 \cdot 2^d$ is always a nonresidue and the claim holds.

If $d$ is even, then $2^d$ is always a quadratic residue, and so we are only concerned with when $-3$ is a nonresidue.  Since $-1$ is a nonresidue mod $p$ if and only if $p \equiv -1 \pmod{4}$, and 3 is a residue if and only if $p \equiv 1, 11 \pmod{12}$, then for $p \equiv 5 \pmod{6}$, either $-1$ is a nonresidue and 3 a residue (when $p \equiv 11 \pmod{12}$) or $-1$ is a residue and 3 a nonresidue ($p \equiv 5 \pmod{12}$).

If $d$ is odd, then we wish $-6$ to be a quadratic nonresidue, and in addition to the case $p \equiv 23 \pmod{24}$ given above, we may also take $p \equiv 13 \pmod{24}$, for which 2 is a nonresidue while $-1$ and 3 are residues. This concludes the proof.
\end{proof}

\begin{proof}[Proof of Corollary \ref{fixedprime}]
We make use of Lemma \ref{quadres}.

If $p \equiv 3, 5 \pmod{8}$, then 2 is a quadratic nonresidue modulo $p$, and so among $-2^e$ and $-3 \cdot 2^e$, half of the values will be nonresidues modulo $p$ and the hypotheses of Theorem \ref{twopowers} are satisfied.

If $p \equiv 7 \pmod{8}$, then $-1$ is a nonresidue mod $p$ and 2 is a residue, so $-2^e$ is always a nonresidue mod $p$.

If $p \equiv 1 \pmod{8}$, then $-1$ and 2 are residues mod $p$, but if $p \equiv 17 \pmod{24}$ then 3 is a nonresidue. Hence $-3 \cdot 2^e$ is always a nonresidue as well.

\end{proof}

\begin{remark}
Note that, consistently with Conjecture \ref{allprimes}, some congruences also appear to exist for primes $p \equiv 1 \pmod{24}$, even though we cannot establish them with the arguments of this paper. For instance, numerical calculations suggest that $f_1^5$ is $73^2$-even with base 1110.
\end{remark}

\begin{proof}[Proof of Corollary \ref{p5cor}]
Here $d$ is even and we are considering $p=5$, so the hypotheses of clause $1^\prime$ of Theorem \ref{primeclasses} are satisfied.  We have $t = 2^{2k}+3$, and seek the value of $-\left(2^{2k-3}3_{25}^{-1} + 2_{25}^{-3}\right) \pmod{25}$.  But observe that $3_{25}^{-1} \equiv -8 \pmod{25}$. Hence
$$r \equiv -\left(2^{2k-3}3_{25}^{-1} + 2_{25}^{-3}\right) \equiv -\left(2^{2k-3}(-1)2^3 - 3\right) \equiv 2^{2k}+3 \equiv t \pmod{25},$$
as desired.

For $p=3$, simply note that $-1 \equiv 2^3 \pmod{9}$ and the calculations follow.
\end{proof}

\begin{proof}[Proof of Theorem \ref{dminus1}]
It is known by a result of Radu \cite{RaduNoEvens}, completing work by Subbarao, Ono, and other authors, that there exists no progression $An+B$, $A \neq 0$, for which $p(An+B) \equiv 0 \pmod{2}$ for all $n$.

We have that $$f_1^{2^d - 1} = \frac{f_1^{2^d}}{f_1} \equiv \frac{f_{2^d}}{f_1}.$$ But then all coefficients $c_{2^d-1}(n)$ of $f_1^{2^d - 1}$ up to $n = 2^d-1$ must match the parity of the partition number $p(n)$.  Now consider any progression $An+B$, $A \neq 0$, as $d$ grows.  If this is always even, then an indefinitely long initial segment of $c_{2^d-1}(An+B)$ must also be even, but by Radu, eventually $p(An+B)$ contains an odd entry, and hence so does $c_{2^d-1}(An+B)$.
\end{proof}

\begin{proof}[Proof of Theorem \ref{t21}]  We observe that $t=21 = 1 + 4 + 16$ is the smallest number not representable by two terms as in Theorem \ref{primeclasses}, and therefore a different proof is required.

The claim follows from establishing the congruence $$q \sum_{n=0}^\infty c_{21}(7n) q^{8n} \equiv q f_8^3 + q^{49} f_{56}^{21}.$$

The actual proof is a standard exercise in the theory of modular forms, so we omit the details.  For all relevant machinery, see the background section of \cite{KZ} and employ the $U(7)$ operator from \cite{OnoWeb} on $q^7 f_8^{21}$.
\end{proof}

\section{Future directions}

We recall our ``master conjecture'' for this series of papers, which posits an ``all or nothing'' behavior for the parity of eta-quotients.

\begin{conjecture}[\cite{KZ}, Conjecture 4]\label{mainconj}
Let $F(q)=\sum_{n\ge 0} c(n)q^n$ be an eta-quotient, shifted by a suitable power of $q$ so powers are integral, and denote by $\delta_F$ the odd density of its coefficients $c(n)$. We have:

i) For any $F$, $\delta_F$ exists and satisfies $\delta_F\le 1/2$.

ii) If $\delta_F= 1/2$, then for any nonnegative integer-valued polynomial $P$ of positive degree, the odd density of $c(P(m))$ is $1/2$. (In particular, $c(Am+B)$ has odd density $1/2$ for all arithmetic progressions $Am+B$, $A\neq 0$.)

iii) If $\delta_F< 1/2$, then the coefficients of $F$ are identically 0 (mod 2) on \emph{some} arithmetic progression. (Note that this is not even \emph{a priori} obvious when $\delta_F=0$.)

iv) If the coefficients of $F$ are not identically 0 (mod 2) on \emph{any} arithmetic progression, then they have odd density $1/2$ on \emph{every} arithmetic progression; in particular, $\delta_F= 1/2$.\\(Note that i), ii), and iii) together imply iv), and that iv) implies iii).)
\end{conjecture}

We noted in Theorem \ref{smallt0} that when $G_t(q)$ satisfies the conditions of CMSZ, its density is 0. When $t$ is even, we believe that the converse also holds, in the following strong form. (Aricheta \emph{et al.} also made this conjecture in \cite{ArichetaEtAl}.)

\begin{conjecture} Let $t= t_0 2^{\alpha}$  be even, with $t_0$ odd, $t_0 > 3 \cdot 2^{\alpha}$ (i.e., $G_t(q)$ does not satisfy the conditions of CMSZ). Then $\delta^{(t)} = 1/2$.
\end{conjecture}

When $t$ is odd, we ask the following.
\begin{question} Is it true that $\delta^{(t)} > 0$ for all $t\ge 5$ odd?
\end{question}

In general, when $t \equiv 3 \pmod{4}$, $g_t(n)$ seems much more likely to have congruences. Up to $t=99$, every such $t$ presents some even progression, and we conjecture that the density of all these series is less than 1/2.  We did find that $t=215$ has no even progressions $An+B$ for $0<A < 8000$; its experimental density is 0.499271. Consistently with Conjecture \ref{mainconj}, this might suggest either density 1/2 and no congruences at all, or, if one exists, that the first such congruence is of rather large modulus.

Interestingly, when $t \equiv 1 \pmod{4}$, $t\geq 29$, computations suggest the opposite behavior; namely, we found no congruences at all for $G_t(q)$.  We are prepared to state the following conjecture.

\begin{conjecture} Let $t \equiv 1 \pmod{4}$, $t \geq 29$. Then  there exists no arithmetic progression $An+B$, $A \neq 0$, such that $g_t(An+B)$ is always even.  Moreover (which is equivalent under Conjecture \ref{mainconj}), we have $\delta^{(t)} = 1/2$.
\end{conjecture}

To the best of our knowledge, this conjectural behavior of certain eta-quotients -- namely, their lack of congruences -- has so far only been established for the partition function $p(n)$.

Among candidate progressions for congruences, we note with interest that many $G_t(q)$, for $t \equiv 3 \pmod{4}$, seem to possess even progressions in which the modulus is a power of 2.  An early example is $g_{19}(16n+11) \equiv 0 \pmod{2}$.

Other congruence families that appear to exist, such as $g_5(98n+(42,56,84)) \equiv 0$, do not follow from the type of arguments employed in this manuscript.  We expect many of these are susceptible to the mass-formula analysis of Ballantine, Merca, and Radu. \cite{BallMercRadu}.

The argument 338 shows up in the study of $g_9(n)$ in Theorem \ref{longlist}.  Since no progression $An+B$ is identically even in the partition function, and initial segments of $g_t(n)$ match the partition function for longer intervals as $t \rightarrow \infty$, it follows that any given argument can only show up for a finite number of values of $t$.  For instance, $g_t(10n+j)$ is even for all $n$ and some value of $j$ only for $t=5$ and $t=25$, and no other $t < 39$.  Since the  sequences $p(10n+k)$ contain an odd element for all $0 \leq k \leq 9$ at or before $n=39$, no other $G_t(q)$ possesses such a progression.

For $t$ odd, extensive computations suggest that candidate progressions $An+B$ may all have an even modulus $A$.

\begin{conjecture}\label{evenmodulus} Let $t$ be odd. If $g_t(An+B) \equiv 0 \pmod{2}$ for all $n$, then $A$ is even.
\end{conjecture}


Regarding the pure eta-powers, we offer the following questions and considerations.

The remaining cases of Conjecture \ref{allprimes} not covered by Theorem \ref{primeclasses}, namely when $p \equiv 1 \pmod{24}$, certainly bear investigation. The first such prime is 73, and experimental calculations suggest several values of $t$ are indeed $73^2$-even with various bases, from $f_1^5$ with base 1110 to $f_1^{69}$ with base 4660.

Computations also seem to indicate that many additional progressions exist for various $t$, though they are not provable directly by the arguments of this paper.  A notable example is $f_1^{13}$, which appears to be $p^2$-even for \emph{some} primes 1 mod 6 in addition to the 5 mod 6 proved above; however, it cannot be so for the same reasons as in the proof of Theorem \ref{primeclasses}.

There are also numerous cases where the full behavior of being $p^2$-even does not hold because exactly half of the required progressions are identically even.  For instance, $c_7(49n+B)$ appears to be even for $B \in \{21, 28, 42\}$, which are three instead of six of the required progressions.  Similarly, $c_{195} \left(73^2n+B\right)$ seems even in 36 of the necessary 72 progressions.  A different approach would be necessary to establish these claims, but possibly there is a unifying underlying reason that may be employed to prove an infinite class of such cases.

We conclude by posing the following questions.

\begin{question}
Fixing a prime $p$ not congruent to 1 mod 24, Theorem \ref{primeclasses} yields an infinite class of $t$ for which $f_1^t$ is $p^2$-even, but the class is of exponential growth.  Is $f_1^t$ $p^2$-even for a positive proportion of odd $t$, and is the same true also for a positive proportion of primes $p \equiv 1 \pmod{24}$?
\end{question}

\begin{question}
Fixing $t$, Theorem \ref{primeclasses} gives a positive proportion of primes $p$ for which $f_1^t$ is $p^2$-even as long as $t$ is of the classes covered, but this is a density 0 set in the integers. Hence we ask: is $f_1^t$ $p^2$-even for a positive proportion of primes $p$, for \emph{all} $t$?
\end{question}

\begin{question}
Chen's theorem in \cite{Chenktuple} yields even progressions for all $c_{3k}$, but the moduli are large, with many prime factors. Theorem \ref{primeclasses} applies to fewer $c_t$, but the resulting arithmetic progressions have relatively small moduli $p^2$.  Does there exist an even arithmetic progression, other than for $f_1$ or $f_1^3$, of the form $pn+B$ for some prime $p$?
\end{question}

\begin{problem} Extend the analysis to eta-powers that contain three 1s in their binary expansions.
\end{problem}

\section{Acknowledgements}
We thank an anonymous reviewer for valuable comments on an earlier version of this paper, and the current reviewers for substantive suggestions on presentation and organization.  The second author was partially supported by a Simons Foundation grant (\#630401).

\section{Data}

This paper did not generate or analyze any data.

\section{Conflict of Interest Statement}

On behalf of all authors, the corresponding author states that there is no conflict of interest.

\end{document}